%% file: main.tex
\documentclass[12pt,reqno]{amsart}

\usepackage[utf8]{inputenc}

\input{preamble.tex}
\title{Higher Convexity and Iterated Second Moment Estimates}
\author{Peter J. Bradshaw, Brandon Hanson, and Misha Rudnev}

\begin{document}

\maketitle

\begin{abstract}
    We prove bounds for the number of solutions to
  $$a_1 + \dots + a_k = a_1' + \dots + a_k'$$ over $N$-element sets of reals, which are sufficiently convex or near-convex. A near-convex set will be the image of a set with small additive doubling under a convex function with sufficiently many strictly monotone derivatives. We show, roughly, that every time the number of terms in the equation is doubled, an additional saving of $1$ in the exponent of the trivial bound $N^{2k-1}$ is made, starting from the trivial case $k=1$. In the context of near-convex sets we also provide explicit dependencies on the additive doubling parameters. 
  
  Higher convexity is necessary for such bounds to hold, as evinced by sets of perfect powers of consecutive integers. We exploit these stronger assumptions using an idea of Garaev, rather than the ubiquitous Szemer\'edi-Trotter theorem, which has not been adapted in earlier results to embrace higher convexity.
  
  As an application we prove small improvements for the best known bounds for sumsets of convex sets under additional convexity assumptions.

\end{abstract}

\section{Introduction}

Let $A$ be a finite subset of $\R$.
Additive combinatorics is concerned with quantifying ``additive structure'' in $A$. 
The size of the sum or difference set 
\[ A\pm A := \{a_1 \pm a_2: a_1,a_2 \in A\} \]
is one such measure of this structure.

In this paper, we will primarily be interested in the stronger second moment quantities. The additive energy is defined as
\[ E(A) := |\{ (a_1,a_2,a_1',a_2') \in A^4: a_1 + a_2 = a_1' + a_2' \}|.\]

More generally
\begin{define}
Let $A_1, \dots, A_k$ be sets of reals. Define
\begin{equation}\label{eq:energy} T(A_1,\dots,A_k) := |\{ (a_1,\dots,a_k,a_1',\dots,a_k') \in \prod_{i=1}^k A_i \times \prod_{i=1}^k A_i: a_1 + \dots + a_k = a_1' + \dots + a_k' \}|. \end{equation}

Here and henceforth $\prod_{i=1}^k A_i$ means Cartesian product, which will also be denoted as $A^k$ when all $A_i=A$.
In the latter case we also write
\[ T_k(A) := T(\underbrace{A,\dots,A}_{k \text{ times}}). \]
For the special case when there are only two summands, we write
$E(A):=T_2(A)$ and $E(A,B):= T(A,B)$.

\end{define}

In this paper, we will specifically explore the additive structure of convex sets and near-convex sets. 
A set $A := \{a_1 < \dots < a_{N} \}$ is called convex if the sequence of adjacent differences $a_{i+1}-a_i$
is strictly monotone for all $1\leq i \leq N-1$. 

If $|A|=N$, then trivially $N^k\leq T_k(A)\leq N^{2k-1},$ and we are interested in proving non-trivial upper bounds for $T_k(A)$. We have the heuristic in mind that convex sets are not very additively structured, so having ``enough'' convexity should push these upper bounds, ideally, close to the lower bound $N^k$. The fact that more and more convexity should be required for higher $k$ is illustrated by setting $A$ to be the set of consecutive perfect powers.

Many of today's state of the art results concerning additive properties of convex sets have been obtained via a particular version of the Szemer\'edi--Trotter theorem, which assumes that one can write a convex set $A$ in the form $A=f([N])$ where the function $f$ has strictly monotone derivative. From now on we just say ``monotone'', meaning ``strictly monotone''. Here $[N]=\{1,\ldots,N\}$, the integer interval. Indeed, spline interpolation can be used to produce such an $f$, so there is no loss of generality in assuming $A=f([N])$. 

Using the Szemer\'edi--Trotter theorem to study convex sets was pioneered by Elekes, Nathanson and Ruzsa \cite{ElekesNathansonRuzsa2000}. However, they only use a special case (where the point set is a Cartesian product), which can be proved by an elementary {\em lucky pairs} method. The reader who is experienced in combinatorial geometry may recognize this as cell-partitioning, but with some tweaking which exploits additive structure. The lucky pairs terminology has been adopted from J. Solymosi; the argument goes back, in particular, to \cite{SolymosiTardos2007}, where it is shown that in the context of point sets which are Cartesian products, the Szemer\'edi-Trotter theorem admits an elementary proof, rather than the original \cite{St83} or subsequent general case proofs.

In this paper, we define a similar notion of lucky pairs, however without any use of the Szemer\'edi-Trotter theorem or geometric incidence arguments. 
Instead, we develop the idea of Garaev from \cite{Garaev2000}, which underlies his elementary proof of the following energy bound. This bound was previously established by Konyagin \cite{Konyagin2000} via the Szemer\'edi-Trotter theorem.

\begin{thm}[Konyagin--Garaev]
\label{garaevs energy bound}
Let $A$ be a 
convex set of $N$ elements.
Then
\[E(A) \ll N^{5/2}. \]
\end{thm}

\begin{notation}
We use Vinogradov's symbol extensively. We write $X \ll Y$ to mean that $X \leq CY$ for some absolute constant $C$, and $X \lesssim Y$ to mean that $X \leq C_1Y(\log Y)^{C_2}$ for some absolute constants $C_1,C_2$.
For several proofs herein, a suppressed constant $C:=C(k)$ may depend on the dimension of the problem $k$. These dependencies can be easily calculated explicitly in the proofs, however we will not, since we canonically think of $k$ as fixed and small compared to the other parameters in our results.
\end{notation}

\medskip
Inductively, one can define higher convexity for finite sets of reals, beginning by saying a $0$-convex set is a set, written in monotone order. For $s\geq 1$, a set $A$ with $N$ elements is $s$-convex if the set of neighbouring differences $\{a_{i+1}-a_i,\,1\leq i\leq N-1\}$ is $(s-1)$-convex. 
Under this definition, a $1$-convex set is simply a convex set.
A telescoping argument implies that if $A$ is $s$-convex, then for any $1\leq h<N$, the set 
$$\Delta_h A:= \{a_{i+h}-a_i,\,1\leq i\leq N-h\}$$ is $(s-1)$-convex.
At times, we may implicitly assume that the $h$-difference set $\Delta_h A$ has $N$ elements (rather than $N-h$), but this will not be of any consequence.

Higher convexity was recently investigated by Roche-Newton and the last two authors \cite{HansonRoche-NewtonRudnev2020} aiming at proving ad infinitum with $k$, growth of $k$-fold sumsets of sufficiently convex sets. This paper develops a different, although not unrelated and also elementary approach to moments $T_k(A)$. Corollary \ref{cor:size} and estimate \eqref{est:card} in Corollary \ref{corollary energy small doubling intro} imply, at least on the qualitative level, the main results in \cite{HansonRoche-NewtonRudnev2020}. Hence in some sense, this paper is a sequel to \cite{HansonRoche-NewtonRudnev2020}. Garaev's method, explored in this paper, has been brought to our attention by the work \cite{Olmezov2020} by Olmezov.

Garaev's argument uses only the following weaker implication of the convexity of $A$, namely that for every $1\leq h<N$, the collection of differences $\{a_{i+h}-a_i\}$ is indeed a set, rather than a multiset. In our forthcoming induction of Garaev's argument for $s$-convex sets, we will iterate this property, applying it to the difference sets $\{a_{i+h}-a_i\}$ as well as to $A$.

We extend Theorem \ref{garaevs energy bound} to $s$-convex sets as follows.

\begin{thm}\label{cor:GG}
For $s\geq 0$, let $k=2^s$ and let $A_1,\dots,A_k$ be $s$-convex sets with $|A_i| \leq N$ for all $1\leq i\leq k$.
Then
\[ T(A_1,\dots, A_k) \ll N^{2^{s+1}-1-s + \alpha_s}, \]
where $\alpha_0 = 0$ and $\alpha_s = \sum_{j=1}^s j2^{-j}.$
\end{thm}

Loosely speaking, Theorem \ref{cor:GG} says that provided our sets are convex enough, each time we double the number of terms $k$ in the energy equation \eqref{eq:energy}, we essentially get a ``saving'' of an additional factor of $N$ off the trivial estimate $N^{2k-1}$ for the quantity $T(A_1,\dots, A_k)$.

We remark that in a recent preprint \cite{Mudgal2021} Mudgal has shown, using an ingenious application of the Balog-Szemer\'edi-Gowers theorem, that cardinality bounds of  \cite{HansonRoche-NewtonRudnev2020} alone enable one to conclude that for an $s$-convex set $A$ one has the estimate
$T_k(A)\leq N^{2k-1 -s +\alpha_k}$, with $\alpha_k\to 0$ as $k\to \infty$, although quite slowly, the characteristic scale being roughly $k\sim 2^{2^s}$, rather than $k=2^s$ here.

Theorem \ref{cor:GG} has a standard sumset implication after an application of the Cauchy--Schwarz inequality, which also illustrates the rough saving of $N$ every time $k$ doubles.
\begin{cor}\label{cor:size}
For $s\geq 1$, let $k=2^s$ and let $A_1,\dots,A_k$ be $s$-convex sets with $|A_i| = N$ for all $1\leq i\leq k$. Then
\[|A_1 \pm A_2 \pm \dots \pm A_k| \gg N^{1 +s - \sum_{j=1}^sj2^{-j}}\,. \]

\end{cor}

We remark that for $s\geq 2$, the proof of Theorem \ref{cor:GG} can be refined to give the improved $\alpha_s = -\frac{2}{13} + \sum_{j=1}^s j2^{-j}$ and slightly more for higher values of $s$. This is follows from us bounding  $\frac{2^s-1}{2^s}$ by $1$ in the induction proof of Theorem \ref{cor:GG} plus the fact that the induction can start at $s=1$, where we have Shkredov's \cite{ShkredovHigherEnergies2013} estimate $E(A)\lesssim |A|^{32/13}$, see Theorem \ref{th:sumset bounds} below.  However, this is not the focus of the theorem and perturbs the exposition so we only comment on the modifications needed to admit this improvement.

However, since we will use the explicit bounds for $k=4$ in the last section of this paper, we state the improved result, according to the remark above.
\begin{thm}\label{th:t4}
If $A_1,A_2,A_3,A_4$ are $2$-convex sets all of size $N$, then 
\[ T(A_1,A_2,A_3,A_4) \lesssim N^{4+24/13} \qquad \text{and} \qquad |A_1\pm A_2\pm A_3\pm A_4| \gtrsim N^{2+2/13}\,. \]
Moreover, for $A_1,A_2,A_3,A_4$ being $s>1$-convex and $1\leq r\leq N^{3}$, one has
\begin{equation}\label{e:needed}
|\{x:\,r_{A_1\pm A_2\pm A_3\pm A_4}(x)\geq r \}|\lesssim \frac{N^4}{r^{7/3}}E_{s-1}\,,   
\end{equation}
where 
$$E_{s}=\sup_{B\mbox{ {\small $s$-\text{convex}, $|B|=N$}}} E(B).$$
\end{thm}

\medskip

We proceed towards formulating our results concerning near-convex sets. It is easy to see (we show it explicitly in the forthcoming Lemma \ref{lem:convconv}) that if $A = f([N])$ where $f$ is a $C^s(\R)$ function with  monotone derivatives $f^{(0)}, f^{(1)},\dots, f^{(s)}$, then $A$ is $s$-convex. A function $f$ with this property is henceforth referred to as an $s$-convex function.

Using $s$-convex functions is not necessary to prove the above results about $s$-convex sets, but they provide a natural way to generalise to {\em near-convex sets.}
We say that a set $A=f(B)$ is {\em near-convex} (more specifically $K-$ near-convex) if $f$ is an $s$-convex function and the set $B$ is such that
$$
|B-B+B| \leq K|B|\,.
$$
The parameter $K$ will be referred to as the doubling constant associated with $B$.

Our main result, Theorem \ref{thm long asymmetric intro}, reflects the maxim that ``convex functions destroy additive structure'', and that the more convex a function, the more it destroys additive structure.
Its proof arises from generalising Garaev's method to longer sums and more convex sets.
In light of our main theorem, and one can view Theorem \ref{cor:GG} as its corollary by setting $B_1 = \dots = B_k =[N].$ \footnote{For this to be exactly true, one needs to check that all $s$-convex sets are of the form $f([N])$ for some $s$-convex function $f$, for $s>1$.}



\begin{thm}\label{thm long asymmetric intro}
Let $B_1, \dots B_{k}$ be any sets with $|B_i| = N$, $|B_i+B_i-B_i| = K_i N$ for all $1\leq i \leq k$.
With $s\geq 0$ and $k = 2^s$, let $A_i = f_i(B_i)$ for some $s$-convex functions $f_1, \dots, f_k$. 
Then we have
\[ T(A_1, \dots, A_{k})\ll \left (\prod_{i=1}^k K_i^{2-(2+2s-2\alpha_s)2^{-s}} \right) \cdot N^{2^{s+1}-1-s + \alpha_s}, \]
where $\alpha_0 = 0$ and $\alpha_s = \sum_{j=1}^s j2^{-j}.$
\end{thm}

Due to the generality of Theorem \ref{thm long asymmetric intro}, it yields useful corollaries. The following follows by setting all the $B_i$ and all the $f_i$ to be the same.

\begin{cor}\label{corollary energy small doubling intro}
Let $B$ be any set with $|B| = N$ and $|B+B-B| = KN$. If $A:=f(B)$ where $f$ is an $s$-convex function and $k = 2^s$, then 
\[ T_k(A) \ll K^{2^{s+1} - 2 - 2s + 2\alpha_s} \cdot N^{2^{s+1}-1-s + \alpha_s}, \]
where $\alpha_0 = 0$ and $\alpha_s = \sum_{j=1}^s j2^{-j}.$

Thus
\begin{equation}
    \label{est:card}
|\underbrace{A \pm A \pm \dots \pm A}_{k \text{ times}}| \gg K^{-2^{s+1} + 2 + 2s - 2\alpha_s} N^{1 +s - \alpha_s}\,. 
\end{equation}
\end{cor}

It should be mentioned that such bounds which depend on doubling constants can be used to obtain sum-product-type results, pioneered by Erd\"os and Szemer\'edi \cite{ErdosSzemeredi1983}; for longer sums and products see \cite{BourgainChang2004}. We will not however explicitly discuss sum-product phenomena further.
Other sum-product type results in the context of convex sets can be seen in recent work of Stevens and Warren \cite{StevensWarren2021}.

We also prove the following asymmetric energy bound:

\begin{thm}\label{short asymmetric intro}
Let $B,C$ be sets with $|B| = N$, $|B+B-B| = KN$ and $|C| = L$.
If $A := f(B)$ for some convex function $f$, then
\[ E(A,C) \ll K^{1/2} N L^{3/2}. \]
\end{thm}
Our result improves the following result, which was previously the best-known and which follows from a straightforward extension of Konyagin's Szemer\'edi--Trotter proof of Theorem \ref{garaevs energy bound}.
The improvement is in the dependence on $K$. 

\begin{thm}\label{weaker short energy}
Let $B,C$ be sets with $|B| = N$, $|B-B| = KN$ and $|C| = L$.
If $A := f(B)$ for some convex function $f$, then
\[ E(A,C) \ll K N L^{3/2}. \]
\end{thm} 

Indeed, the improved Theorem \ref{short asymmetric intro} is sharp when $|A| = |C|$.
Let $A = C = f(B)$ where $B = \{x^2: x \in [N]\}$ and $f(x):= \sqrt{x}$. Then we get $E(A,C) = K^{1/2}|A||C|^{3/2} = N^3$.

\medskip
The utility of incorporating higher convexity into the main results is as follows. For $1$-convex sets, the best known estimates for $T_{k}(A)$, have been derived in \cite{IosevichKonyaginRudnevTen2006} using induction and Szemer\'edi-Trotter bounds, namely
\begin{equation}\label{energy IKRT}
    T_{k}(A)\ll N^{2k - 2+ 2^{-(k-1)}}\,.
\end{equation}
This in particular implies that $|kA|\gg N^{2-2^{-(k-1)}}$.
This estimate cannot be improved beyond $N^2$, as evinced by the first $N$ squares, which form a $1$-convex (but not $2$-convex) set. Consequently, the energy bound \eqref{energy IKRT} is almost best-possible. One would naturally expect better estimates for more convex sets, but the choice techniques based on the Szemer\'edi--Trotter theorem have not been adapted to allow for nontrivial estimates. A different elementary technique has been developed in \cite{HansonRoche-NewtonRudnev2020} to account for the following growth of particular sumsets:
\[ |2^s A - (2^s-1)A| \gg |A|^{s+1}. \]

In this paper we show that Garaev's idea underlying his proof of Theorem \ref{garaevs energy bound}, allows for upper bounds for the quantities $T_{2^s}$, which are in line with the sumset bounds in \cite{HansonRoche-NewtonRudnev2020}.

Note that bounds for $T_k(A)$ can be immediately recycled into $L^{2k}$ bounds for exponential sums over $A$ (see \cite{IosevichKonyaginRudnevTen2006} and the references contained therein). 

In the last section of this paper we use our estimates for the quantities $T_3(A)$ and $T_4(A)$ to get small improvements on the best known sumset and energy estimates for sets which are $s$-convex, for $s\geq 2$. Although our quantitative improvements are quite modest (and most likely not best possible within the technology we present) they do break the ice in some sense, for previously used methods based on the Szemer\'edi-Trotter theorem did not enable one to benefit by using higher convexity.

\subsection{Organisation of this paper}
Section \ref{prelims} contains the notation used herein which relates to energy and difference sets. It also briefly discusses the relevant properties of convex sets and convex functions.
Section \ref{methods} introduces the Garaev argument and a generalised definition of ``lucky pairs''. The key result is Proposition \ref{prop lucky pairs most general}, which forms the framework for proving the main results of the paper.
Sections \ref{proofs1} and \ref{proofs2} respectively contain the proofs of the main results in the convex and near-convex cases.
Section \ref{applications} is devoted to proving new sumset results for $s$-convex functions where $s\geq 2$.

\section{Preliminaries and Notation}\label{prelims}

Throughout this paper, if $x \in A_1 + \dots + A_k$, then \[r_{A_1+ \dots + A_k}(x) := |\{(a_1,\dots,a_k)\in\prod_{i=1}^k A_i: x = a_1 + \dots + a_k\}|. \]
Also for $r\geq 1$, we define 
\begin{equation}\nonumber X_r:=\{x \in A_1 + \dots + A_k: r\leq r_{A_1+ \dots + A_k}(x) < 2r\}, \end{equation}
the set of {\em $r$-rich sums} in a given sumset. Whenever we use (extensively) the notation $X_r$, context will make it clear what sumset it is contained in.

In this language we can express the energy in either of the following ways:
\[T(A_1,\dots,A_k) = \sum_{x \in A_1+ \dots + A_k} r_{A_1+ \dots + A_k}^2(x),  \]
or
\[T(A_1,\dots,A_k) \approx \sum_{r \text{ dyadic}} r^2|X_r|.  \]

For a convex set $A$ with $N$ elements, and $1\leq h< N$ its $h$-difference set is $$\Delta_h A:=\{a_{i+h}-a_i:\,i\in[N-h]\}\,.$$
We will not explicitly deal with the fact that $|\Delta_h A|<|A|=N,$ but rather simply add $h$ extra elements to $\Delta_h A$ ad hoc.

We also define ``discrete derivatives''. Given a function $f$, let its $h$-derivative be \[ \Delta_h f(x):=f(x+h)-f(x). \]
\begin{lem}
If $f$ is an $s$-convex function, then for any $h$, $\Delta_h f$ is an $(s-1)$-convex function.
\end{lem}
\begin{proof}
We use induction on $s$. Suppose $f$ is $1$-convex.
We have 
\[ \Delta_h f(x) := f(x+h)-f(x) = \int_x^{x+h} f'(y)dy. \]
Since $f'$ is monotone, it follows that $\Delta_h f$ is also monotone, and hence $0$-convex. 

Next assume the statement holds for $(s-1)$-convex functions. 
Let $f$ be an $s$-convex function.
By definition, this implies that $f'$ is an $(s-1)$-convex function.
The induction hypothesis implies that $\Delta_h(f')$ is an $(s-2)$-convex function. But since $\Delta_h(f') = (\Delta_h f)'$, it follows that $\Delta_h f$ is $(s-1)$-convex, completing the induction. 
\end{proof}

A one-way relationship between $s$-convex functions and $s$-convex sets is intuitive and is summarised in the following lemma.

\begin{lem}\label{lem:convconv}
If $f$ is an $s$-convex function, then $f([N])$ is an $s$-convex set. 
\end{lem}

\begin{proof}
By induction on $s$. If $f$ is a $0$-convex function, then $f([N])$ is clearly ordered as a $0$-convex set.

Assume the statement holds for $(s-1)$-convex functions. Let $f$ be an $s$-convex function. Then $(\Delta_1 f)(x) := f(x+1)-f(x)$ is an $(s-1)$-convex function. By the induction hypothesis, $(\Delta_1 f)([N])$ is an $(s-1)$-convex set, which proves that $f([N])$ is an $s$-convex set, completing the induction. \end{proof}
We expect the converse to the statement of Lemma \ref{lem:convconv} to be true, as it is for $s=1$. However, proving this may require interpolation techniques beyond the scope of this paper.





\section{Methods}\label{methods}

We begin by presenting a version of Garaev's proof of Theorem \ref{garaevs energy bound}. This proof is essentially synthesised from its exposition by  Olmezov \cite{Olmezov2020}, with an additional observation that convexity can be used more sparingly, which enables one to extend the estimate for $E(A)$ to $E(A,B)$, where $A$ is a convex set and $B$ is any set. This is based on replicating estimate \eqref{simple X_r with B} below, known earlier via the Szemer\'edi-Trotter theorem. 

In the forthcoming argument, we only need the following property of a  {\em convex} set 
$$A=\{a_1, a_2,\ldots a_N\},
$$
written in increasing order: For each $h<N$, the differences $a_{i+h}-a_i,\,i=1,\ldots,N-h$ are all distinct.

\begin{proof}[Proof of Theorem \ref{garaevs energy bound}]
We are estimating the number of solutions to
\begin{equation}\label{2energy}
a_{i_1} + a_{j_1} = a_{i_2} + a_{j_2} \,:(a_{i_1},a_{j_1},a_{i_2},a_{j_2})\in A^4\,.
\end{equation}

Consider some $x \in X_r$.
Recall that this means $r\leq r_{A+A}(x)<2r$ where $r_{A+A}(x)$ is the number of realisations of $x$ as a sum of two elements in $A$.
 Write
\[ x = a_{i_1} + a_{j_1} = \dots = a_{i_r} + a_{j_r}, \]
with  $i_1<i_2<\ldots<i_r$. Since $a_{i_u} + a_{j_u} = x$ for all $u$, we also have $j_1>j_2>\ldots >j_r$. 
At the cost of a multiplicative constant, we may also assume that $j_u \geq i_u$ for all $u$.
It follows that 
\[ \sum_{u=1}^{r-1} (i_{u+1}-i_u) \leq N \qquad \text{and} \qquad \sum_{u=1}^{r-1} (j_{u}-j_{u+1}) \leq N. \]
By the pigeonhole principle, at least $3r/4$ of the summands in each sum cannot exceed $4N/r$. This implies that there is a set of indices $U \subset [r-1]$ with $|U|\geq r/2$ such that for every $u\in U$, $i_{u+1}-i_u \leq 4N/r$ and $j_{u}-j_{u+1}\leq 4N/r$. For $u \in U$, we say the pair $(a_{i_u},a_{j_u}), (a_{i_{u+1}},a_{j_{u+1}})$ is a \emph{lucky pair}, so there are at least $r/2$ lucky pairs.

Since each lucky pair gives rise to a solution to the energy equation \eqref{2energy},
there are  least $r/2$ distinct solutions of the equation
$$
a_{i_1+h_1} - a_{i_1} = a_{i_2+h_2} - a_{i_2}\,,
$$
where $i_1,i_2\in [N]$ and $1\leq h_1,h_2\leq 4N/r$.

By considering all $x \in X_r$, it follows that 
\[
r|X_r| \ll (N/r)^2 \max_{1\leq h_1,h_2\ll N/r}|
\{(i_1,i_2) \in [N]^2:\, a_{i_1+h_1} - a_{i_1} = a_{i_2+h_2} - a_{i_2}\}|\,. 
\]
Now comes the only part of the argument where we use the convexity of $A$: given $h_1$, all differences $a_{i_1+h_1} - a_{i_1}$ are distinct, hence for any fixed $h_1$ and $h_2$, we have trivially that
\begin{equation}\label{e:induct_triv}
|\{(i_1,i_2) \in [N]^2:\, a_{i_1+h_1} - a_{i_1} = a_{i_2+h_2} - a_{i_2}\}|\leq N\,.
\end{equation}
It follows that
\begin{equation}\label{simple X_r bound}
   |X_r| \ll N^3/r^3. 
\end{equation}
We write $E(A) = \sum_{r \text{ dyadic}} r^2|X_r|$ and, for some parameter $r_*$ to be chosen, use the trivial bound $|X_r| \ll N^2/r$ for $r \leq r_*$ and estimate \eqref{simple X_r bound} for $r > r_*$.
Choosing the optimal $r_* = N^{1/2}$ yields the desired
\[ E(A) \ll N^{5/2}. \qedhere \]
\end{proof}
We remark that since the lucky pairs argument itself involves solely the pigeonhole principle and no assumptions on the set $A$, the above proof generalises immediately to the case of $E(A,B),$ where $A$ is convex and $B$ any set. Bound \eqref{simple X_r bound} becomes
\begin{equation}\label{simple X_r with B}
|X_r| \ll |A||B|^2/r^3\,,
\end{equation}
with $X_r$ now being the set of $r$-rich sums in $A+B$. 
Indeed, the only necessary changes to the proof are that now $h_2$ pertain to the set 
$$B=\{b_1,b_2,\ldots, b_{|B|}\},\,$$ so that $1\leq h_2\ll |B|/r$, and the trivial bound \eqref{e:induct_triv} is replaced by $|B|$. This takes into account that given $h_2$, the quantities $b_{i_2+h_2} - b_{i_2}$ are not necessarily all distinct. What matters is that $a_{i_1+h_1} - a_{i_1}$ are all distinct.

Hence, Garaev's method enables one to obtain the standard corollary of estimate \eqref{simple X_r with B}, which is usually proved using the Szemer\'edi-Trotter theorem.

\begin{cor} If $A$ is a convex set, then for any $B$,
$$\begin{array}{rrl} E_3(A,B):= \sum_x r^3_{A\pm B}(x) &\ll& |A|B|^2\log|A|\,, \\
E_{1+p}(A,B):= \sum_x r^{1+p}_{A\pm B}(x) &\ll& |A||B|^{1+p/2}\,,\;\mbox{ for } 1<p<2\,.\end{array}
$$ \label{cor:ST}
\end{cor}
Earlier expositions of Garaev's method appear to overlook the fact that it generalises easily to embrace two different sets $A$ and $B$, owing to an overreliance on convexity in the proof.

\medskip
In order to generalise Theorem \ref{garaevs energy bound} to the quantity $T_{2^k}(A)$, we need to generalise the concept of lucky pairs from above. In order not to repeat ourselves, we do it in the general setting, suitable for all the results in this paper. In the convex set setting, $B_1, \ldots, B_k$ below are all just the interval $[N]$.
In the near-convex setting, the full generality of Definition \ref{defn lucky pairs most general} and Proposition \ref{prop lucky pairs most general} will be needed.


\begin{define}
[Lucky Pairs]\label{defn lucky pairs most general}
For $1 \leq i \leq k$, suppose $B_i$ is a finite set of real numbers, $g_i$ is a monotone function and $A_i = g_i(B_i)$. 
Given any $r$, where $r^{1/(k-1)} \ll |B_i+B_i-B_i|$ for all $1\leq i\leq k$,
let \[X_r=\{x\in A_1+\cdots+A_k:r\leq r_{A_1+\cdots+A_k}(x)< 2r\}\] be the $r$-rich sums in $A_1 + \dots + A_k$.
Suppose $P:=(b_1,\dots,b_k)$ and $P':=(b_1',\dots,b_k')$ are \emph{distinct} points, each belonging to $\prod_{i=1}^k B_i$. For $x \in X_r$, we say $(P,P')$ forms a \emph{lucky pair} associated with $x$ if the following two conditions hold.
\begin{enumerate}
    \item The pair $(P,P')$ gives rise to a solution to the energy equation for the sum $x$. That is,
    \[ g_1(b_1) + \dots + g_k(b_k) = x = g_1(b_1') + \dots + g_k(b_k'). \]
    \item In all coordinates, there are not many elements of $B_i+B_i-B_i$ between $P$ and $P'$. That is, if $n_{B_i}(b,b')$ is the number of elements of $B_i+B_i-B_i$ lying in $(b,b']$ (or $(b',b]$ if $b>b'$), then
    \begin{equation}\label{lucky pair condition}
    n_{B_i}(b_i,b_i') \ll |B_i+B_i-B_i|/r^{1/(k-1)},    
    \end{equation}
    for all $1\leq i \leq k$.
\end{enumerate}
\end{define}

\begin{rmk}
In \eqref{lucky pair condition}, we will always be treating the upper bound on $n_{B_i}(b_i,b_i')$ as an integer.
This is why we insist that $r^{1/(k-1)} \ll |B_i+B_i-B_i|$ in Definition \ref{defn lucky pairs most general}.
For all the results in this paper, this condition will hold trivially, so it will not be discussed further.
\end{rmk}

\begin{prop}\label{prop lucky pairs most general}
Let $r,k\geq 1$ and for $1 \leq i \leq k$, suppose $B_i$ is a finite set of real numbers, $g_i$ is a monotone function and $A_i = g_i(B_i)$. Let \[X_r=\{x\in A_1+\cdots+A_k: r \leq r_{A_1+\cdots+A_k}(x)< 2r\}\]
be the $r$-rich sums in $A_1 + \dots + A_k$. Then for each $x\in X_r$, there are $\gg r$ lucky pairs associated with $x$.
\end{prop}

We will need the following lemma in the proof:

\begin{lem}\label{hyperplane lemma}
Suppose we have a $k$-dimensional box in $\R^k$ (a Cartesian product of $k$ orthogonal intervals) which is composed of $r^k$ smaller (nonidentical) boxes (or cells) in an $r\times \dots \times r$ grid. Then any generic hyperplane $H$ (not parallel to any one-dimensional edge of the box) can pass through at most $kr^{k-1}$ cells.
\end{lem}

\begin{proof}
By translation and scaling, we may assume that the origin is one of the corners of the box, the facets of the box are all parallel to coordinate hyperplanes and that the hyperplane $H$ is of the form $X_1 + \dots + X_k = C$ for some constant $C$.

We can index each cell by a $k$-tuple $(e_1, \dots, e_k)$ which denotes its position among the cells on each axis, starting from the origin. 
Now for each cell in which at least one of the $e_i$ is $1$, we define its associated \emph{diagonal} as the set of cells with indices
$(e_1+a, \dots, e_k+a)$ for  $0\leq a \leq r - \max_i e_{i}$.

There are $kr^{k-1}$ such diagonals which cover all the cells, and $H$ intersects each diagonal in at most one cell, completing the proof.
\end{proof}

\begin{proof}[Proof of Proposition \ref{prop lucky pairs most general}]
For each $i$, partition $B_i + B_i - B_i$ into $r^{1/(k-1)}/4$ intervals, each containing $4|B_i+B_i-B_i|/r^{1/(k-1)}$ elements. Since $B_i \subset B_i+B_i-B_i$, this also partitions the elements of $B_i$. Doing this for each $i$ partitions $\prod_i B_i$ into boxes and hence, since the $g_i$ are all monotone functions, also partition $\prod_i A_i$ into boxes (or cells). 

Now consider some $x \in X_r$. Each solution to
\[ x = g_1(b_1) + \dots + g_k(b_k) \]
corresponds to a point $(g_1(b_1), \dots, g_k(b_k))$ on the hyperplane
\[ x= X_1 + \dots + X_k. \]
By Lemma \ref{hyperplane lemma} this hyperplane can pass through at most $(k/4^{k-1})\cdot r \leq r/2$ cells and the hyperplane has $r$ points on it. By the pigeonhole principle, there must be $\gg r$ pairs of points which lie together in the same cell. By construction, these are lucky pairs, which completes the proof.
\end{proof}

    \section{Proof of Theorem \ref{cor:GG}}\label{proofs1}

Despite the fact that Theorem \ref{cor:GG} is essentially a less general version of Theorem \ref{thm long asymmetric intro}, we present its proof separately to illustrate exactly how much convexity is needed.

\begin{proof}[Proof of Theorem \ref{cor:GG}]
Let us denote the desired universal bound for $T_k(A_1,\ldots,A_k)$ as $$\mathcal T_k:=\mathcal T_k(N)=\sup T_k(A_1,\ldots,A_k)\,$$ where the supremum is taken over all $k$-tuples of $s$-convex sets of size $N$.
Let $A_i := \{ a^{(i)}_1 < \ldots < a^{(i)}_N \}$ for each $1 \leq i \leq k$. 
We are counting solutions to the equation 
\begin{equation}\label{long convex energy equation}
a^{(1)}_{e_1} + \dots + a^{(k)}_{e_k} = a^{(1)}_{e_1'} + \dots + a^{(k)}_{e_k'},
\end{equation}
for some indices $e_1, \ldots, e_k, e_1',\ldots, e_k' \in [N]$.

The proof is by induction on $s$ where $k=2^s$, the base case $s=0$ being trivial: the number of solutions of
$$
a=a':\quad a,a'\in A_1
$$
is at most (in fact precisely) $N$.

Let us assume that in the equation \eqref{long convex energy equation} no two terms $a^{(i)}_{e_i}$ and $a^{(i)}_{e'_i}$ are the same for $i=1,\ldots,k.$ More precisely, suppose that such non-degenerate solutions to equation \eqref{long convex energy equation} constitute at least half of the quantity $T_k(A_1,\ldots,A_k).$ Indeed, the number of degenerate solutions is at most
\[\sum_{j=1}^kT_{k-1}(A_1,\ldots,A_{j-1},A_{j+1},\ldots,A_k)\]
and by freezing all but $k/2$ of the variables on each side and applying Cauchy-Schwarz, we get 
$$ T_k(A_1,\ldots,A_k)\ll_k N^{k-1} \mathcal T_{k/2}\, \ll \mathcal{T}_k.$$
Thus if the degenerate solutions constituted more than half of the upper bound, the proof would be complete.

Recall that $X_r$ is the set of $r$-rich sums. For each $x \in X_r$, we apply Proposition \ref{prop lucky pairs most general} with $B_1 = \ldots = B_k = [N]$ and $g_i(y):=a^{(i)}_y$, for all $1\leq i \leq k$. By considering all lucky pairs arising from any $x \in X_r$, one obtains
$$
r|X_r|\ll \# \text{ solutions to } \eqref{long convex energy equation}\,,
$$
where $|e_i - e_i'| \ll N/r^{1/({k-1})}$ for all $1\leq i \leq k$.
We now choose the $h_i:=e_i-e_i'$ for $1\leq i \leq k$ which maximise the number of  solutions to \eqref{long convex energy equation}.

Notice that for each $h_i$, $\Delta_{h_i} A_i:=\{a^{(i)}_{e_i'+h_i} - a^{(i)}_{e_i'}\}$ is an $(s-1)$-convex set (not a multiset) and has $\leq N$ elements. Here we have used that $h_i$ is non-zero, which is a consequence of the non-degeneracy assumption.
We can subtract all the elements on the right-hand side of \eqref{long convex energy equation}, and since there are $N^k/r^{k/(k-1)}$ ways altogether of choosing $e_1-e_1',\ldots,e_k-e_k'$, it follows that
$$
r|X_r|\ll  \frac{N^k}{r^{{k}/({k-1})}}\;\cdot \# \text{ of solutions to } a_1 +\ldots +a_k = 0
$$
where $a_i \in \Delta_{h_i} A_i$  for  $1\leq i\leq k$.
Rearranging the terms of the above equation so there are $k/2$ terms on each side of the equation and using Cauchy-Schwarz, one can then apply the induction hypothesis to obtain
\begin{equation}|X_r| \ll   \frac{N^k}{r^{(2k-1)/(k-1)}} \cdot \mathcal T_{k/2}\,.
\label{est:xr}\end{equation}

Using $T_k(A_1,\ldots,A_k) = \sum_{r \text{ dyadic}} r^2|X_r|$, we optimise in $r$ by taking, for some $r_*$ to be determined, the trivial bound 
$r_* N^k$ for $r\leq r_*$, and the dyadic sum with \eqref{est:xr} over the values of $r\geq r_*$. Thus
$$
T_k(A_1,\ldots,A_k) \ll r_* N^k +  \frac{N^k}{r_*^{1/(k-1)}} \mathcal T_{k/2}\,.
$$
Taking the optimal choice of
$$
r_* = \mathcal{T}_{k/2}^{1-1/k}\,,
$$
we get 
\[T_k(A_1,\ldots,A_k)\ll N^k \mathcal T_{k/2}^{1-1/k}
= N^{2^s} \cdot N^{(2^s-s+\alpha_{s-1})(1-2^{-s})} \ll N^{2^{s+1}-1-s+\alpha_s}.  \]
This closes the induction and completes the proof.
\end{proof}


\begin{rmk}
The step where we apply Cauchy--Schwarz is a variation on the well-known procedure to prove that
\[ E(A,B) \leq E(A)^{1/2}E(B)^{1/2}. \]
\end{rmk}

As alluded to in the introduction, we can refine this approach to obtain a slightly better bound, specifically a smaller value of $\alpha_s$. 
If we assume that $s\geq 2$ then $s=2$ becomes the base case of the induction. Using the bound \eqref{est:xr}, $T(A_1,A_2,A_3,A_4)$ can be bounded in terms of $E(A)$ where $A$ is a $1$-convex set of size $N$.
Estimating $E(A)$ using Theorem \ref{garaevs energy bound}
produces the improvement $\alpha_s = -\frac{1}{8} + \sum_{j=1}^s j2^{-j}$. Using instead Shkredov's stronger bound \cite{ShkredovHigherEnergies2013}
\[ E(A) \ll N^{32/13}, \]
gives the further improvement $\alpha_s = -\frac{2}{13} + \sum_{j=1}^s j2^{-j}$.

In the above proof, $s$-convexity is only used in one place.  
Since all the $A_i$ are $s$-convex, the sets $\Delta_{h_i}A_i$ are $(s-1)$-convex.
In particular, this implies that $\Delta_{h_i}A_i$ will always be a set rather than a multiset which is essential when iterating the argument. 



\section{Proofs of Theorems \ref{short asymmetric intro} and \ref{thm long asymmetric intro}}\label{proofs2}

In this section, we focus on the results pertaining to sets with small additive doubling. The following lemma is included to clarify the key step for generalising our earlier results. 
\begin{lem}\label{small doubling lemma}
Let $D:=\{d_1 < d_2 < \dots < d_{|D|}\}$ be the positive differences in $B-B$. If $b, b' \in B$ with $b > b'$ and  $n_B(b,b') \leq Z$, then $b - b' \leq d_Z$. In other words, if $n_B(b,b') \leq Z$ then there are at most $Z$ possible values that $b-b'$ can take. 
\end{lem}

\begin{proof}
If not then $b-b' = d_Y$ where $Y>Z$. But then 
\[ b' < b' + d_i \leq b, \]
for $i = 1,\dots, Y$. Thus there are at least $Y > Z$ elements of $B+B-B$ in $(b',b]$, contradicting that $n_B(b',b) \leq Z$.
\end{proof}

\begin{proof}[Proof of Theorem \ref{short asymmetric intro}]
Let $C := \{c_1 <\dots < c_{L}\}$. 
For each $x \in X_r$, we apply Proposition \ref{prop lucky pairs most general} with $k=2, B_1 = B, B_2 = [L]$ and functions $g_1(t) := f(t), g_2(t) := c_t$.
By considering all $x \in X_r$, this implies that the total number of lucky pairs from $r$-rich sums is $\gg r|X_r|$. Since lucky pairs give rise to solutions to the energy equation, it follows that
\begin{equation}\label{simple energy small doubling asymmetric}
r|X_r| \ll \# \text{ solutions to } f(b_1) - f(b_2) = c_{e_{2}} - c_{e_1},
\end{equation}
where $n_B(b_1,b_2) \ll KN/r$ and $|e_{2}-e_{1}| \ll L/r$. 
It follows from Lemma \ref{small doubling lemma} that there are at most $KN/r$ possible values for $|b_2-b_1|$. 

After fixing $b_1-b_2, e_{2}-e_{1}$ and $c_{e_{1}}$ in \eqref{simple energy small doubling asymmetric}, which can be done in $KNL^2/r^2$ ways, the energy equation admits at most one solution since $f_d(x):= f(x+d) - f(x)$ is a monotone function.

It follows that $|X_r|\ll KNL^2/r^3$. Since \[E(A,B)\ll  \sum_{r \text{ dyadic}} r^2|X_r|\ll r_*\sum_{\substack{r \text{ dyadic}\\r\leq r_*}} r|X_r|+\sum_{\substack{r \text{ dyadic}\\r> r_*}} \frac{KNL^2}{r}\ll r_*NL+\frac{KNL^2}{{r_*}}\]
we get that, upon choosing $r_*=(KL)^{1/2}$,
\[ E(A,B) \ll K^{1/2}NL^{3/2}. \qedhere \]

\end{proof}

\begin{proof}[Proof of Theorem \ref{thm long asymmetric intro}]
As in the proof of Theorem \ref{cor:GG}, denote the desired universal bound for $T_k(A_1,\ldots,A_k)$ as $$\mathcal T_k(N;K_1, \dots, K_k)\,.$$
We are counting solutions to the equation 
\begin{equation}\label{long asymmetric energy equation}
f_1(b_1) + \dots + f_k(b_{k}) = f_1(b_1') + \dots + f_k(b_{k}'),
\end{equation}
where $b_i, b_i' \in B_i$ for all $i$.

The proof is by induction on $s$ where again the base case $s=0$ is trivial: the number of solutions of
$$
f_1(b)=f_1(b'):\quad b,b' \in B_1
$$
is at most $N$.

Let us assume that for each solution to \eqref{long asymmetric energy equation} no two terms $f_i(b_i)$ and $f_i(b_i')$ are equal for any $i=1,\ldots,k.$ More precisely, suppose that such non-degenerate solutions to equation \eqref{long asymmetric energy equation} constitute at least half of the quantity $T_k(A_1,\ldots,A_k).$ For otherwise, as in the proof of Theorem \ref{cor:GG}, using a trivial upper bound and the induction hypothesis, we would have 
$$T_k(A_1,\ldots,A_k)\ll N^{k-1} \mathcal T_{k/2}(N;K_{\iota_1}, \dots, K_{\iota_{k/2}}) \, \ll \mathcal{T}_k(N;K_1,\dots,K_k),$$
where $K_{\iota_1}, \dots, K_{\iota_{k/2}}$ are the largest $k/2$ terms among all the $K_i$. This would complete the proof immediately.

As previously seen, $X_r$ contains the sums $x \in A_1 + \dots + A_k$ with $r \leq r_{A_1 + \dots + A_k}(x) < 2r$.
For each $x \in X_r$, we now apply Proposition \ref{prop lucky pairs most general} with $g_{i}(b) := f_i(b)$ for $1\leq i\leq k$.
This obtains
\[ r|X_r| \ll \# \text{ solutions to } \eqref{long asymmetric energy equation}, \]
where $n_{B_i}(b_{i},b_{i}') \ll K_i N/r^{1/(k-1)}$ for all $1 \leq i \leq k$.

We now choose the $h_i:=b_i-b_i'$ for $1\leq i \leq k$ which maximise the number of  solutions to \eqref{long asymmetric energy equation}, and then rearrange to obtain
\begin{equation}\label{iterated long asymmetric energy equation}
(\Delta_{h_1}f_{1})(b_1') + \dots + (\Delta_{h_{k/2}}f_{k/2})(b_{k/2}') = (\Delta_{h_{k/2+1}}f_{k/2+1})(b_{{k/2}+1}) + \dots + (\Delta_{h_k}f_{k})(b_{k}).
\end{equation}
By Lemma $\ref{small doubling lemma}$, there are at most $\prod_{i=1}^k(K_i N)/r^{k/(k-1)}$ ways altogether of choosing the $b_i-b_i'$, so we have
\begin{equation}\nonumber
r|X_r| \ll \frac{\prod_{i=1}^k(K_i N)}{r^{k/(k-1)}} \cdot \# \text{ solutions to } \eqref{iterated long asymmetric energy equation}.
\end{equation}

Applying Cauchy--Schwarz proves that the number of solutions to \eqref{iterated long asymmetric energy equation} is bounded above by
\begin{equation}\label{to use induction hypothesis on}
T((\Delta_{h_1} f_1)(B_1),\dots,(\Delta_{h_{k/2}}f_{k/2})(B_{k/2}))^{1/2} T((\Delta_{h_{{k/2}+1}}f_{k/2+1})(B_{k/2+1}),\dots,\Delta_{h_{k}}(f_k)(B_k))^{1/2}.
\end{equation}

Since all the functions $\Delta_{h_i} f_i$ are $(s-1)$-convex, the induction hypothesis upper bounds \eqref{to use induction hypothesis on} by
\[ \mathcal T_{k/2}(N;K_1, \dots, K_{k/2})^{1/2} \mathcal T_{k/2}(N;K_{k/2+1}, \dots, K_k)^{1/2}, \]
whence 
\begin{equation}\label{est:xr small doubling}
    |X_r| \ll \frac{\prod_{i=1}^k(K_i N)}{r^{(2k-1)/(k-1)}} \cdot\mathcal T_{k/2}(N;K_1, \dots, K_{k/2})^{1/2} \mathcal T_{k/2}(N;K_{k/2+1}, \dots, K_k)^{1/2}.
\end{equation}

Using $T_k(A_1,\ldots,A_k) = \sum_{r \text{ dyadic}} r^2|X_r|$, we optimise in $r$ by taking, for some $r_*$ to be determined, the trivial bound 
$r_* N^k$ for $r\leq r_*$, and the dyadic sum with \eqref{est:xr small doubling} over the values of $r\geq r_*$. Thus
$$
T_k(A_1,\ldots,A_k) \ll r_* N^k +  \frac{\prod_{i=1}^k(K_i N)}{r_*^{1/(k-1)}} \mathcal T_{k/2}(N;K_1, \dots, K_{k/2})^{1/2} \mathcal T_{k/2}(N;K_{k/2+1}, \dots, K_k)^{1/2}\,.
$$
Taking the optimal choice of
$$
r_* = \left(\prod_{i=1}^k K_i^{1-\frac{1}{k}}\right) \cdot \mathcal T_{k/2}(N;K_1, \dots, K_{k/2})^{(\frac{1}{2}-\frac{1}{2k})} \mathcal T_{k/2}(N;K_{k/2+1}, \dots, K_k)^{(\frac{1}{2}-\frac{1}{2k}})\,,
$$
it is elementary to check that
\[T_k(A_1,\ldots,A_k)\ll r_*N^k \ll \left (\prod_{i=1}^k K_i^{2-(2+2s-2\alpha_s)2^{-s}} \right) \cdot N^{2^{s+1}-1-s+\alpha_s}.\]
This closes the induction and completes the proof.
\end{proof}

Similar to Theorem \ref{cor:GG}, we can refine this approach to obtain a bound with a slightly smaller (at least by $1/8$) value of $\alpha_s$ for $s\geq 2$. 

\smallskip
The above proofs apply to sums of length $k=2^s$, where we start the induction with the trivial estimate for $s=0$.
One can also easily develop similar inductions that start with the quantity $T(A_1,A_2,A_3)$ and formulate analogues of Theorems \ref{cor:GG} and \ref{thm long asymmetric intro} for $k=3\cdot 2^s$.
We leave this to the interested reader, concluding this section by stating the base case $k=3$, since it will be once used in the next section.

\begin{thm}\label{length 3 asymmetric energy thm}
If $A_1,A_2,A_3$ are $2$-convex sets with $N$ elements. Let $X_r$ be the set of $r$-rich sums from $A_1+A_2+A_3$. Then
\begin{equation} \label{e:3-popular}
|X_r| \ll \frac{ N^{14/3} }{r^{5/2}}\,.
\end{equation}
In particular
\[ T_3(A_1,A_2,A_3) \ll N^{4+\frac{1}{9}}. \]
\end{thm}

\begin{proof}
By the familiar lucky pairs argument
$$
r|X_r|\ll N^3/r^{3/2}\cdot S_{A+B=C}\,,
$$
where $S$ is the maximum number of solutions to
$$ a+b=c: \quad a\in A,\,b\in B,\,c\in C.
$$
for some $1$-convex sets $A,B,C$ with $|A|=|B|=|C|=N$. 
It remains to show that $S_{A+B=C} \ll N^{5/3}$.

Consider the $r_0$-rich sums $a+b \in A+B$ and recall the corresponding bound \eqref{simple X_r with B}. Combining with a trivial bound, we get
\[ S_{A+B=C} \ll |A||B|^2/r_0^2 + |C|r_0 = N^3r_0^{-2} + Nr_0, \]
and optimising in $r_0$ obtains
\[S_{A+B = C}\ll N^{5/3}\,. \qedhere
\]

\end{proof}

\section{Applications to convex sumsets}\label{applications}
We use our estimates to improve the state of the art sumset bounds for convex sets. The best known bounds to date are respectively due to Schoen and Shkredov \cite{SchoenShkredov2011}, Rudnev and Stevens \cite{RudnevStevens2020} and Shkredov \cite{ShkredovHigherEnergies2013}, and are summarised below.
\begin{thm}\label{th:sumset bounds} If $A$ is convex, then
\begin{align*} 
|A-A|&\;\gtrsim \;|A|^{8/5\;=\;1.6}\,\\
|A+A|&\;\gtrsim \;|A|^{30/19\;\approx\; 1.579}\,,\\
E(A) &\;\lesssim \;|A|^{32/13\;\approx\; 2.4615}\,.
\end{align*}\end{thm}

We can get small improvements of all these bounds for $s$-convex sets, with $s\geq 2$. These estimates rely on using our new bounds apropos of the quantity $T_4(A)$  in Theorem  \ref{th:t4}, as well as $T_3(A)$ in \eqref{e:3-popular}. We will incorporate them into existing methods developed by Shkredov and collaborators  (see, e.g. \cite{SchoenShkredov2011}, \cite{ShkredovHigherEnergies2013}, \cite{MurphyFPMS}), which rely extensively on the use of the third moment \begin{equation}E_3(A):=\sum_x r_{A-A}^3(x)\ll N^3\log N\,,\nonumber\end{equation}
by \eqref{simple X_r bound} (the same bound would normally be attributed to the use of Szemer\'edi-Trotter theorem).

We have attempted to make exposition in this section prerequisite-free. Hence, observe that after resummation the quantity $E_3(A)$ (not to be confused with $T_3(A)$) has the following meaning. If $[a,b,c]$ denotes an equivalence class of triples $(a,b,c)\in A^3$ by translation, with $r([a,b,c])$ triples therein, and $[A^3]$ denotes the set of these equivalence classes, then
\begin{equation}\label{eq:3eq}
E_3(A) = \sum_{x\in [A^3]} r^2(x)\,.
\end{equation}

Let $A$ be an $s$-convex set. When $s=2$, the improvement comes from fetching the equation, accounting for the moment $T_4(A)$, for which we have estimate \eqref{e:needed}, where for the quantity $E_{s-1}$ one can substitute the energy bound from Theorem \ref{th:sumset bounds} for $1$-convex sets.
Furthermore, for $s>2$ this process can then be iterated to obtain incrementally better energy bounds for more convex sets, the iterations rapidly converging. We note that even the simpler energy estimate $E(A)\ll N^{5/2}$ for $1$-convex sets would already improve the estimates of Theorem \ref{th:sumset bounds} for $2$-convex sets. We present only estimates for $2$-convex sets in the next theorem; the small improvements for more convex sets can be found in the forthcoming proof. 

\begin{thm}\label{th:sumset bounds new} If $A$ is a $2$-convex set with sufficiently large\footnote{If $N$ is sufficiently large, then decimal approximation of the exponents enable one to replace the $\lesssim, \gtrsim$ symbols by, respectively, $\leq,\geq$.} size $N$, then
\begin{align*}
|A-A|&\;\gtrsim\; N^{1+ 151/234 \;\approx\; 1.645}\,,\\
|A+A|&\;\gtrsim\; N^{1+ 229/309\;\approx \;1.587}\,,\\
E(A) & \; \leq \; N^{2.4554}\,.
\end{align*}\end{thm}

\begin{proof} 
We begin with the technically least demanding bound for the set $A-A$, where $A$ is $s$-convex for $s\geq 2$. Consider a tautology on triples $(a,b,c)\in A^3:$
$$
(a-c)=(a-b)+(b-c)\,.
$$
If $|A-A|=KN$ and $D$ be the set of popular differences with $\gg N/K$ realisations, by the pigeonhole principle the above tautology is valid for $\gg N^3$ triples $$(a,b,c)\in A^3:\, b-c,\,a-b\in D\,.$$
On the other hand, considering equivalence classes of triples $(a,b,c)$ by translation and using \eqref{eq:3eq} in combination with Cauchy-Schwarz, one has
\begin{equation}\label{cs:dif}\begin{aligned}
N^6 & \ll E_3(A) \; |\{d_1+d_2 \in A-A:\,(d_1,d_2)\in D^2\}| \\ & \lesssim
N^3 (K/N)^{2} \; |\{a_1-a_2+a_3-a_4 = d:\,(a_1,\ldots,a_4) \in A^4,\, d\in A-A\}|\,.
\end{aligned}\end{equation}

Using the H\"older inequality 
and \eqref{e:needed} with dyadic summation yields 
\begin{equation}\label{e:hold}|\{a_1-a_2+a_3-a_4 = d:\,(a_1,\ldots,a_4)\in A^4,\, d\in A-A\}|
\lesssim (KN)^{4/7} (N^4E_{s-1})^{3/7}\,,
\end{equation}
with $E_{s-1}$ as in \eqref{e:needed}. Thus
$$
N^{19/7} \lesssim K^{18/7}E_{s-1}^{3/7}\,.
$$
Using Shkredov's bound for $E_{s-1}$ yields, for $2$-convex $A$:
$$
K\gtrsim N^{\frac{151}{234}}\,.
$$
If $A$ is more than $2$-convex, one can asymptotically use the forthcoming bound \eqref{e:energyconv} for $E_{s-1}$, which improves the exponent for $|A-A|$ just by slightly over $.001$. Namely, if $A$ is sufficiently convex and $N$ is large enough, it follows that
$$
K\geq N^{.646}\,,
$$
the decimal approximation having accounted for replacing $\lesssim$ by $\leq$ for sufficiently large $N$ and $s$.

Furthermore, to bound $|A+A|$ we use a slightly more involved pigeonholing technique which is exposed in more detail in  \cite{RudnevStevens2020}. 
Suppose, $|A+A|=KN,$ define $P$ as a set of sums with $\gtrsim N/K$ realisations, and $D$ as a popular set of differences by energy (we will use energy to connect the difference set with the sum set). Namely $D$ is defined as follows.

By the dyadic pigeonhole principle, there exists $D\subseteq A-A, $ and a real number $1\leq \Delta<|A|$, such that for every $d\in D,\,\Delta\leq r_{A-A}(d)<2\Delta,$
and on top of this 
\begin{equation}
\label{en:pop}
E(A) \lesssim |D|\Delta^2.
\end{equation}
Moreover, by \eqref{simple X_r bound} one has $|D|\Delta^3 \ll |A|^3$ so
\begin{equation}\label{en:pops}
\Delta \lesssim |A|^3/E(A)\,.
\end{equation}

Then (after possibly passing to a large subset of $A,$ see \cite[Proof of Theorem 5]{RudnevStevens2020}) the analogue of the argument underlying the above estimate \eqref{cs:dif} becomes the tautology
$$
(a+b)-(b+c) = a-c \in D,\,a+b, \,b+c \in P\,,
$$
with the number of  triples $(a,b,c)\in A^3$ realising this being  $\gtrsim |A||D|\Delta.$ 

Hence, using \eqref{eq:3eq} and Cauchy-Schwarz as in \eqref{cs:dif} yields
$$\begin{aligned}
(|A||D|\Delta)^2 & \ll E_3(A) \;|\{s_1-s_2 = d:\,(s_1,s_2,d)\in P^2\times D\}| \\ & \lesssim
N^3 (K/N)^{2} \;|\{a_1+a_2-a_3-a_4 = d:\,(a_1,\ldots,a_4)\in A^4,\, d\in D^3\}|\\ &\lesssim K^2N^{19/7}E_{s-1}^{3/7} |D|^{4/7}\,,
\end{aligned}
$$
after using the H\"older inequality and \eqref{e:needed}, as in \eqref{e:hold}.

Multiplying both sides by 
$\Delta^{6/7}\lesssim \left(\frac{N^3}{E(A)}\right)^{6/7}$ (see \eqref{en:pops}) to balance the powers of $|D|$ and $\Delta$, so one can use 
$|D|\Delta^2\gtrsim E(A)$ (see \eqref{en:pop}), yields
$$
E(A)^{16/7} \lesssim K^2E_{s-1}^{3/7} N^{23/7}\,.
$$
Substituting Shkredov's bound for $E_{s-1}$ and using the standard Cauchy-Schwarz bound
$$
K\geq \frac{N^3}{E(A)}
$$
yields
$$
K\gtrsim N^{\frac{229}{390}}\,.
$$

Once again, if $A$ is more than $2$-convex, one can asymptotically use the forthcoming bound \eqref{e:energyconv} for $E_{s-1}$,  in which case
$$
K\gtrsim N^{\frac{16}{27}\; \approx\; 0.592}\,,
$$
the decimal approximation having accounted for replacing $\lesssim$ by $\leq$ for sufficiently large $N$ and $s$.

We now turn to the energy $E(A)$ estimate, where the analysis ends up being somewhat more involved. For the reader's convenience we briefly recall the key steps of Shkredov's spectral (alias operator) method we use, following, e.g. \cite{ShkredovHigherEnergies2013} (for an overview of the method see  \cite{OlmezovSpectral}). The operator method is used to replace the lower bounds from easy tautologies that enabled estimate \eqref{cs:dif} and its analogue in the bounds for $|A\pm A|$ derived above.

Once again, let $D$ be the set of popular differences by energy, satisfying \eqref{en:pop}, \eqref{en:pops}.

Identifying $D$ with its characteristic function, consider the quantity
$$S:= \sum_{a,b,c\in A} D(a-b) D(b-c) r_{A-A}(a-c) \,.
$$
The quantity $S$ takes triples $(a,b,c)\in A^3$ for which $a-b,b-c$ are in the ``popular'' set $D$, and counts each one the number of times the difference $a-c$ repeats itself. The spectral method enables one to get a lower bound on $S$, to be compared with the upper bound we will again obtain by \eqref{eq:3eq} and Cauchy-Schwarz.

Let us view $D$ as a $|A|\times |A|$ symmetric boolean matrix with $1$ at the position $(a,b)$ if $a-b\in D$ and $0$ otherwise. Similarly $r_{A-A}$ can be seen as a square symmetric matrix $R$ (where $R_{ij}:= r_{A-A}(i-j)$), which in addition is non-negative definite (checking this is tantamount to rearrangement of the energy equation, see e.g. \cite{ShkredovHigherEnergies2013}). Thus $S={\rm tr}\, DDR.$

Let $\mu_1$ be a positive eigenvalue of $D$ with the largest modulus and normalised eigenvector $\boldsymbol v\geq 0$ with all non-negative entries; this is possible by the Perron-Frobenius theorem. Since $D$ is symmetric, one can estimate 
$$
\mu_1 = \boldsymbol v\cdot D \boldsymbol v \geq \frac{|D|\Delta}{|A|}\,,
$$
replacing  $\boldsymbol v$ by the vector $\frac{1}{\sqrt{|A|}}\boldsymbol 1.$

Since $D$ is symmetric, one can write $D = Q \tilde D Q^{\top}$ so that $\tilde D$ is diagonal with $\mu_1$ in the top left corner, and $\boldsymbol v$ is the first column of orthogonal matrix $Q$. The basis-invariance of trace gives
\[ S = {\rm tr}(\tilde D^2 Q^\top RQ). \]
Noting that $R$ is nonnegative definite, the trace can be bounded from below by the $(1,1)$-entry, whence
\[ S \geq \mu_1^2 \ \boldsymbol v \cdot R\boldsymbol v.\]
Since $\boldsymbol v\geq 0,$ this can be estimated from below by making the matrix $R$ entry-wise smaller, namely replacing it with $\Delta D$. But for the latter matrix, once again, we can replace $\boldsymbol v$ with $\frac{1}{\sqrt{|A|}}\boldsymbol 1$ to get a lower bound $\boldsymbol v\cdot R \boldsymbol v\gtrsim E(A)/|A|$, and hence 
$$
S\gtrsim \frac{(|D|\Delta)^2E(A)}{|A|^3}\,.
$$

On the other hand the quantity $S$, tautologically, is the number of solutions of the equation 
$$
(a-b) + (b-c) = a'-c':\; a-b,\,b-c \in D,\,a',c'\in A\,.
$$
It follows from Cauchy-Schwarz that
$$
S^2\leq E_3(A) \sum_{d_1,d_2\in D} r^2_{A-A}(d_1+d_2)\,,
$$
and since each of $d_1,d_2$ has at least $\Delta$ representations in $A-A$, this means
$$ S^2 \lesssim\; |A|^3\Delta^{-2} \sum_{x}r_{A-A+A-A}(x)r^2_{A-A}(x).$$

We partition $A-A$ into ``rich and poor'' sets $D_1$ and $D_2$, so that for some $\tau$ to be determined, $r_{A-A}(x)\leq \tau,$ for every $x\in D_1$.

We firstly consider the poor differences $D_1$. By the H\"older inequality 
$$ \sum_{x\in D_1} r_{A-A+A-A}(x)r_{A-A}^2(x) \leq 
\left(\sum_{x\in D_1} r_{A-A+A-A}(x)^{7/3}\right)^{3/7}\left(\sum_{x\in D_1} r_{A-A}(x)^{7/2}\right)^{4/7}\,.
$$
From \eqref{e:needed} we have, once again,
$$
\sum_x r_{A-A+A-A}(x)^{7/3}\lesssim N^4 E_{s-1}\,,
$$
and from the definition of $D_1$,
$$\sum_{x\in D_1} r_{A-A}(x)^{7/2}\lesssim N^3 \tau^{1/2}\,.
$$

As for the set $D_2$, we have, from \eqref{simple X_r bound},
$$
|D_2|\leq N^3/\tau^3\,.
$$
Without changing the notation, let us replace $D_2$ by its subset $\{x:\,\tau\leq r_{A-A}(x)<2\tau\}$: this will not have consequences, after a subsequent dyadic summation. Then the quantity to be estimated is 
$$
\sum_{x\in D_2} r_{A-A+A-A}(x)r_{A-A}^2(x) \leq \tau^2 |\{d=a_1+a_2-a_3-a_4:\,d\in D_2;a_1,\ldots,a_4\in A\}|\,.
$$
By the H\"older inequality, this is bounded by
$$
\tau^2 \left(\sum_{x}r_{A+A-A}(x)^{5/2}\right)^{2/5} \left(\sum_{x}r_{A+D}(x)^{5/3}\right)^{3/5}\,.
$$
The first bracketed term is estimated directly using \eqref{e:3-popular}. Moreover, by the second bound of Corollary \ref{cor:ST},
$$
\sum_{x}r_{A+D}(x)^{5/3}\ll N|D|^{4/3}\,.
$$

Putting everything together,
\begin{equation}\nonumber
|D|^4\Delta^6E^2(A) \lesssim |A|^9\left(N^{24/7}E_{s-1}^{3/7}\tau^{2/7} + N^{73/15}\tau^{-2/5}\right)\,.
\end{equation}
Optimising in $\tau$ yields
$$
\tau=N^{151/72}E_{s-1}^{-5/8}\,.
$$
Multiplying both sides by $\Delta^2$, using $E(A)\lesssim |D|\Delta^2$ (see \eqref{en:pop}) on the left and $\Delta\lesssim |A|^3/E(A)$ (see \eqref{en:pops}) on the right yields
$$
E^8(A) \lesssim N^{15+24/7+151/252} E_{s-1}^{1/4}\,.
$$
It remains to substitute an estimate for $E_{s-1}$. If $A$ is $2$-convex we can use Shkredov's bound $E_{s-1}\lesssim N^{32/13}\,$, and we arrive at 
$$
E(A) \lesssim N^{2+1705/3744} \leq N^{2+.4554}\,,
$$
for sufficiently large $N$\,.

One can iterate this bound for higher convexity (namely using it as $E_{s-1}$ if $s=3,$ etc.) and it is easily seen that the iterates converge rapidly. In the limit when $E(A)=E_{s-1}$ in the above calculation one gets
\begin{equation}\label{e:energyconv}
E(A) \lesssim N^{2+127/279}\leq N^{2+.4552}\,.
\end{equation}
\end{proof}

\section*{Acknowledgements}
{ We would like to thank Oliver Roche-Newton for providing more than a welcome advise throughout the process of writing this paper. The First Author would like to thank Oliver Clarke, Charley Cummings and Harry Petyt for a useful discussion about Lemma \ref{hyperplane lemma}. The Second Author is supported by the NSF Award 2001622. The Third Author has been partially supported by the Leverhulme Trust Grant RPG-2017-371.}

\bibliographystyle{plain}
\bibliography{bibliography}

\end{document}

%% file: preamble.tex
\usepackage[a4paper,top=3cm,bottom=2cm,left=3cm,right=3cm,marginparwidth=1.75cm]{geometry}

\usepackage{amsmath}
\usepackage{amssymb}
\usepackage{amsthm}
\usepackage{enumitem}
\usepackage{dsfont}
\usepackage{verbatim}
\usepackage{xcolor}

\theoremstyle{definition}
\newtheorem{define}{Definition}[section]

\theoremstyle{plain}
\newtheorem{thm}{Theorem}[section]
\newtheorem{lem}{Lemma}[section]
\newtheorem{prop}{Proposition}[section]
\newtheorem{cor}{Corollary}[section]

\theoremstyle{remark}
\newtheorem*{rmk}{Remark}

\newtheorem*{notation}{Notation}

\newcommand{\R}{\mathbb{R}}

\usepackage[myheadings]{fullpage}
\usepackage{fancyhdr}
\usepackage{lastpage}
\usepackage{graphicx, wrapfig, setspace}
\usepackage[T1]{fontenc}
\usepackage[english]{babel}